\documentclass{easychair}

\usepackage{doc}

\usepackage{amssymb,amsfonts,stmaryrd}
\usepackage[shortlabels]{enumitem}
\newcommand{\IN}{\mathbb{N}}
\newcommand{\IZ}{\mathbb{Z}}
\newcommand{\IQ}{\mathbb{Q}}
\newcommand{\IR}{\mathbb{R}}
\newcommand{\IC}{\mathbb{C}}
\newcommand{\IF}{\mathbb{F}}

\newcommand{\Mm}{\mathcal{M}}
\newcommand{\Ll}{\mathcal{L}}
\DeclareMathOperator{\ring}{ring}

\DeclareMathOperator{\Th}{Th}

\newcommand{\laurent}[1]{(\!(#1)\!)}
\theoremstyle{plain}

\newtheorem{thm}{Theorem}%[subsection]
\newtheorem*{thm*}{Theorem}
\newtheorem{prop}[thm]{Proposition}
\newtheorem{lem}[thm]{Lemma}

\theoremstyle{definition}
\newtheorem{defi}[thm]{Definition}
\newtheorem{rem}[thm]{Remark}
\newtheorem{expl}[thm]{Example}

\title{Undecidability of expansions of Laurent series fields by cyclic discrete subgroups}
\author{Leo Gitin}
\institute{
  University of Oxford,
  Oxford, U.K.\\
  \email{leo.gitin@maths.ox.ac.uk}
 }
\authorrunning{Leo Gitin}
\titlerunning{Undecidability of expansions of Laurent series fields}

\begin{document}

\maketitle

\begin{abstract}
In 1987, Pheidas showed that the field of Laurent series $\IF_q\laurent{t}$ with a constant for the indeterminate $t$ and a predicate for the natural powers $\{t^n \mid n > 0\}$ of $t$ is existentially undecidable. We show that the same result holds true if $t$ is replaced by any element $\alpha$ of positive $t$-adic valuation.
\end{abstract}

\paragraph{Introduction.}
\label{sect:introduction}

Hilbert's Tenth Problem asks for an \textit{algorithm} that, given a polynomial $f(X_1, \ldots, X_n)$ with integer coefficients, will determine whether or not it has a root in integers $\IZ$, see \cite{Hilbert24,Hilbert24-english}.
Building on previous work by Robinson, Davis, and Putnam, Matiyasevich famously showed that no such algorithm exists \cite{book:Matiyasevich}. Hilbert's Tenth Problem can be equivalently phrased as asking whether or not the positive existential theory $\Th_{\exists^+}(\IZ)$ in the first-order language of rings $\Ll_{\ring} = \{0,1,+,\cdot\}$ is decidable \cite[1.1]{survey:Koenigsmann} (in what follows, we will omit the symbols of $\Ll_{\ring}$ when speaking of ring structures). In the context of model theory, it is both natural to consider other structures $\Mm$ that may differ from $\IZ$ and to extend the family of sentences that we look at (e.g. the existential theory $\Th_{\exists}(\Mm)$ or the entire theory $\Th(\Mm)$). Many classical results in logic and model theory subsume answers to decidability questions.

Before Matiyasevich's negative solution to Hilbert's Tenth Problem, it was already known by Gödel's work on his Incompleteness Theorems \cite{Goedel} that the full first-order $\Ll_{\ring}$-theory $\Th(\IZ)$ is undecidable. In the 1930s and 1950s, Tarski \cite{Tarski31, Tarski51} determined the $\Ll_{\ring}$-theories of the real and complex fields $\IR$, $\IC$ (the archimedean local fields) and consequently showed that both are decidable. Ax and Kochen \cite{paper:Ax-Kochen-I} studied the model theory of non-archimedean local fields, i.e., $p$-adic fields $K$ (finite field extensions of the $p$-adic numbers $\IQ_p$) and Laurent series fields
\[
    \IF_q\laurent{t} = \left\{\sum_{i = -k}^{\infty} a_i t^i\, \Big|\, a_i \in \IF_q,\, k \in \IZ \right\}
\]
over finite fields $\IF_q$ with $q = p^n$ elements, $p$ a prime number. It follows from their work that the theory $\Th(K)$ of any $p$-adic field $K$ is decidable. Whether or not the Laurent series fields are decidable, is a major open question in the model theory of valued fields. In 2016, Anscombe and Fehm \cite{Anscombe-Fehm} made substantial progress towards this question by proving the decidability of the existential theory $\Th_{\exists}(\IF_q\laurent{t})$ of Laurent series fields. For other recent results in this direction, we refer to Anscombe, Dittmann, and Fehm \cite{Anscombe-Dittmann-Fehm,Dittmann-Fehm}. 

It is natural to consider the structures mentioned above in expansions of the language of rings. Van den Dries \cite{vdDries85} considered the real ordered field with a new predicate for $2^{\IZ}$, the cyclic multiplicative subgroup generated by 2. He proves the surprising result that $(\IR,2^{\IZ})$ is decidable by showing quantifier elimination in a natural expansion of $(\IR,2^{\IZ})$. This still holds if 2 is replaced by a recursive real number $\alpha > 1$. In the same paper, van den Dries asks if his results can be generalised to the structure $(\IR,2^{\IZ},3^{\IZ})$. In 2010, Hieronymi \cite{Hieronymi} gave a negative answer: for two real numbers $\alpha, \beta > 1$ satisfying $\alpha^{\IZ} \cap \beta^{\IZ} = \{1\}$, the theory $\Th(\IR,\alpha^{\IZ},\beta^{\IZ})$ is undecidable. Expansions of $\IQ_p$ by discrete cyclic (multiplicative) subgroups have been studied by Mariaule \cite{Mariaule_dec,Mariaule_undec}. He proves that for $\alpha \in \IQ_p$ of positive $p$-adic valuation $v_p(\alpha) > 0$, the theory $\Th(\IQ_p,\alpha^{\IZ})$ is decidable, whereas $\Th(\IQ_p,\alpha^{\IZ},\beta^{\IZ})$ is undecidable whenever $v_p(\beta) > 0$ and $\alpha^{\IZ} \cap \beta^{\IZ} = \{1\}$.
Ax already knew (unpublished) that $\Th(\IF_q\laurent{t},t^{\IZ})$ is undecidable. An elementary proof was given by Becker, Denef, and Lipshitz \cite{Becker-Denef-Lip}. Later,
a considerable strengthening was obtained by Pheidas \cite{Pheidas}. This is particularly interesting, as not much is known about these fields from the point of view of (un)decidability. He shows:
\begin{thm*}[Pheidas] \label{intro:pheidas_result}
Let $P = \{t^n \mid n > 0\}$ be the set of powers of the indeterminate $t$. Then $\Th_{\exists}(\IF_q\laurent{t}, t, P)$ is undecidable.
\end{thm*}
Note that by virtue of Anscombe and Koenigsmann \cite{Anscombe-Koenigsmann}, who show that the valuation ring $\IF_q\llbracket t \rrbracket$ in $\IF_q\laurent{t}$ is existentially $\Ll_{\ring}$-definable without parameters, it follows moreover that $\Th_{\exists}(\IF_q\laurent{t}, t, t^{\IZ})$ is undecidable (observe that $t^{\IZ} \cap \IF_q\llbracket t \rrbracket = P \cup \{1\}$). We generalise this theorem to arbitrary cyclic discrete subgroups of $\IF_q\laurent{t}$, i.e., subgroups generated by an element $\alpha$ of positive $t$-adic valuation $v_t(\alpha)$.
\begin{thm*}
Let $\alpha \in \IF_q\laurent{t}$ be an element with $v_t(\alpha) > 0$. Then the existential theory of the structure $(\IF_q\laurent{t},\alpha,\alpha^{\IZ})$ is undecidable.
\end{thm*}

See Remark \ref{rem:generalise} for a more general formulation. Another way of viewing $\alpha^{\IZ}$ is to think of it as the image of a homomorphism from the value group $\IZ$ into the multiplicative group $\IF_q\laurent{t}^{\times}$. When $v_t(\alpha) = 1$, such a homomorphism is called a cross-section.

\paragraph{Pheidas' work.}

Pheidas proves his theorem in two steps. His key tool is the following (somewhat unusual) relation on natural numbers that goes back to Denef \cite{Denef} and is sometimes called \textit{$p$-divisibility}. We write
\[
    n \mid_p m \text{ if and only if } \exists k \in \IN\ m = n \cdot p^k.
\]
His proof now proceeds as follows.
\begin{enumerate}[(I)]
    \item Prove that $\Th_{\exists}(\IN,0,1,+,\mid_p)$ is undecidable by giving an existential definition of multiplication in this structure and invoking the Matiyasevich/MRDP theorem. \label{idea:1}
    \item Show that the relation $n \mid_p m$ can be effectively coded in $\IF_q\laurent{t}$ by an existential formula via $P = \{t^n \mid n > 0\}$. \label{idea:2}
\end{enumerate}
To generalise from $t$ to arbitrary $\alpha$, we precisely follow Pheidas' strategy. The main content of this note is to explain how Pheidas' coding needs to be modified in this more general context.

Essential to the coding is the unique arithmetic of $\IF_q\laurent{t}$.

\begin{rem}
In characteristic $p$, both the Frobenius map $x \longmapsto x^p$ and the Artin-Schreier map $x \longmapsto x^p - x$ are additive. Moreover, the Frobenius map is an automorphism on the finite field $\IF_q$ and a non-surjective monomorphism on $\IF_q\laurent{t}$ with image
\[
    \IF_q\laurent{t^p} = \left\{\sum_{i = -k}^{\infty} a_{pi} t^{pi}\, \Big|\, a_{pi} \in \IF_q,\, k \in \IZ \right\}.
\]
This is the field of $p^{\text{th}}$ powers in $\IF_q\laurent{t}$.
\end{rem}

\begin{lem} \label{lem:p_div_coding}
Fix an element $\alpha \in \IF_q\laurent{t}$ with $v_t(\alpha) > 0$ not divisible by $p$. We can characterise the relation $n \mid_p m$ for natural $m, n > 0$ as follows:
\begin{equation} \label{eq:p_div_coding}
    n \mid_p m \quad \text{if and only if} \quad m \ge n \wedge \exists a \in \IF_q\laurent{t}\ \alpha^{-m} - \alpha^{-n} = a^p - a.
\end{equation}
\end{lem}
\begin{proof}
Pheidas' proof \cite[Lem. 1]{Pheidas} for $\alpha = t$ goes through in this case. We will use this opportunity to show his beautiful argument.

Assume $n \mid_p m$ holds such that $m = n \cdot p^k$ for some $k \in \IN$. In that case, the element
\[
    a = \alpha^{-n p^{k - 1}} + \alpha^{-n p^{k - 2}} + \ldots + \alpha^{-n}
\]
witnesses the right-hand side of (\ref{eq:p_div_coding}). Conversely, assume that for positive integers $m \ge n$, there exists $a \in \IF_q\laurent{t}$ satisfying $\alpha^{-m} - \alpha^{-n} = a^p - a$. Write $m = m_0 p^{v_p(m)}$ and $n = n_0 p^{v_p(n)}$, where both $m_0, n_0 > 0$ are not divisible by $p$. By the first part of the proof, we can find $b, c \in \IF_q\laurent{t}$ with
\begin{align*}
    \alpha^{-m} - \alpha^{-m_0} & = b^p - b \\
    \alpha^{-n} - \alpha^{-n_0} & = c^p - c.
\end{align*}
Setting $d = a - b + c$, we can combine these three equations to $\alpha^{-m_0} - \alpha^{-n_0} = d^p - d$. If $m_0 = n_0$, we are done since $m \ge n$. Otherwise, we may assume $m_0 \ne n_0$, in which case
\[
    v_t(d^p - d) = v_t(\alpha^{-m_0} - \alpha^{-n_0}) = -v_t(\alpha)\max\{m_0,n_0\}.
\]
We know $v_t(d) < 0$ implies that $v_t(d^p - d)$ is divisible by $p$, which is in contradiction to our assumptions that $v_t(\alpha)$, $m_0$, $n_0$ are not divisible by $p$.
\end{proof}

\begin{rem} \label{rem:power_of_alpha}
Note that (\ref{eq:p_div_coding}) still holds in the case when we can write $\alpha = \beta^{p^k}$, $k \in \IN$, where $v_t(\beta)$ is not divisible by $p$. Indeed, for $m \ge n$, we have
\[
    \exists a \in \IF_q\laurent{t}\ \alpha^{-m} - \alpha^{-n} = \beta^{-m p^k} - \beta^{-n p^k} = a^p - a
\]
iff $n p^k \mid_p m p^k$ iff and only if $n \mid_p m$.
\end{rem}

\paragraph{The general case.}

This characterisation of $\mid_p$ given by (\ref{eq:p_div_coding}) will not work for all possible values of $\alpha$, as we can see by the following counterexample.
\begin{expl} \label{expl:counter}
Consider $p = q = 3$, i.e., the local field $\IF_3\laurent{t}$ and the element
\[
    \alpha = (t^{-3} + 1 + t + t^2)^{-1}
\]
with $v_t(\alpha) = 3$ divisible by $p = 3$. Then $\alpha^{-2} - \alpha^{-1} = a^3 - a$ has a solution in $\IF_3\laurent{t}$,
\[
    a = t^{-2} + t^{-1} - t + t^2 + \sum_{i \ge 0} (-1)^i (-t^{4 \cdot 3^i} + t^{6 \cdot 3^i}),
\]
but the relation $1 \mid_3 2$ does not hold.
\end{expl}
\renewcommand*{\thefootnote}{\fnsymbol{footnote}}
Hence a new observation is needed. For this purpose, we define the following unusual function, which we call the ``$p^{\text{th}}$-powers-omitting $t$-adic valuation'' for lack of a better name.\footnote{Note that, strictly speaking, $\hat v_t$ is not a valuation on $\IF_q\laurent{t}$: it does not satisfy $x = 0 \Longleftrightarrow \hat v_t(x) = \infty$ and it is also not a group homomorphism.}
\renewcommand*{\thefootnote}{\arabic{footnote}}
\begin{defi} \label{def:p_omit_val}
Given $x \in \IF_q\laurent{t}$, written as a Laurent series
\[
    x = \sum_{i = -k}^{\infty} a_i t^i,
\]
define $\hat v_t(x)$ to be the integer
\[
    \hat v_t(x) = \min\{i \mid a_i \ne 0 \wedge p \nmid i\},
\]
and $\hat v_t(x) = \infty$ if this minimum does not exist, i.e., if $x \in \IF_q\laurent{t^p}$.
\end{defi}
Curiously, it captures exactly the kind of algebraic-combinatorial behaviour of $\IF_q\laurent{t}$ that becomes invisible to $v_t$.

\begin{lem} \label{lem:v-hat_t_of_powers}
Assume that $\alpha \in \IF_q\laurent{t}$ is not a $p^{\text{th}}$ power, but $p \mid v_t(\alpha) > 0$. Let $N \in \IN$ be not divisible by $p$. Then
\[
    \hat v_t (\alpha^N) = (N - 1)v_t(\alpha) + \hat v_t(\alpha).
\]
\end{lem}
\begin{proof}
Decompose $\alpha$ as $\alpha = \beta + \gamma$, where $\beta \ne 0$ contains all monomials with exponent divisible by $p$ and $\gamma \ne 0$ contains all monomials with exponent not divisible by $p$. By our assumptions,
\[
    v_t(\beta) = v_t(\alpha) < \hat v_t(\alpha) = \hat v_t(\gamma).
\]
Considering the binomial theorem for $(\beta + \gamma)^N$, we observe that
\[
    \binom{N}{N - 1} \beta^{N - 1} \gamma
\]
must contain the monomial with the smallest exponent not divisible by $p$. Thus
\begin{equation*}
    \hat v_t(\alpha^N) = \hat v_t (N \beta^{N - 1} \gamma) = (N - 1)v_t(\beta) + \hat v_t(\gamma) = (N - 1)v_t(\alpha) + \hat v_t(\alpha). \qedhere
\end{equation*}
\end{proof}

\begin{lem} \label{lem:p_div_coding2}
Fix an element $\alpha \in \IF_q\laurent{t}$ with valuation $v_t(\alpha) = C > 0$ divisible by $p$. Assume additionally that $\alpha$ is not a $p^{\text{th}}$ power, so that $\hat v_t(\alpha^{-1}) = D \in \IZ$. Then for any choice of $N > 0$ satisfying
\[
    N > \frac D C + 1 \quad \text{and} \quad p \nmid N,
\]
we have
\[
    n \mid_p m \quad \text{if and only if} \quad m \ge n \wedge \exists a \in \IF_q\laurent{t}\ \alpha^{-mN} - \alpha^{-nN} = a^p - a
\]
for all $m,n > 0$.
\end{lem}
\begin{proof}
If $n \mid_p m$ holds, we essentially take the same witness $a \in \IF_q\laurent{t}$ as in Lemma \ref{lem:p_div_coding}. As for the converse, let us consider positive integers $m \ge n$ such that there exists $a \in \IF_q\laurent{t}$ with
\[
    \alpha^{-mN} - \alpha^{-nN} = a^p - a.
\]
By repeating the same steps as in the proof of Lemma \ref{lem:p_div_coding}, we can write $m = m_0 p^{v_p(m)}$, $n = n_0 p^{v_p(n)}$ and find $d \in \IF_q\laurent{t}$ such that
\begin{equation} \label{eq:contradiction}
    \alpha^{-m_0N} - \alpha^{-n_0N} = d^p - d.
\end{equation}
We are done if $m_0 = n_0$. So assume without loss of generality that $m_0 > n_0 \ge 1$. Instead of considering the $t$-adic valuation on both sides of (\ref{eq:contradiction}), we look at the $p^{\text{th}}$-powers-omitting $t$-adic valuation instead. By Lemma \ref{lem:v-hat_t_of_powers} and $p \nmid m_0N$, we observe
\begin{equation} \label{eq:val_LHS}
    \hat v_t(\alpha^{-m_0N} - \alpha^{-n_0N}) = -(m_0N - 1)C + D.
\end{equation}
If we evaluate the right-hand side of (\ref{eq:contradiction}), we get
\begin{equation} \label{eq:val_RHS}
    \hat v_t(d^p - d) = \hat v_t(d) \ge v_t(d).
\end{equation}
Since $v_t(d) < 0$, we can use
\[
    pv_t(d) = v_t(d^p - d) = v_t(\alpha^{-m_0N} - \alpha^{-n_0N}) = -m_0NC,
\]
together with (\ref{eq:contradiction}), (\ref{eq:val_LHS}), and (\ref{eq:val_RHS}), to deduce the inequality
\[
    -(m_0N - 1)C + D \ge \frac{-m_0NC}{p}.
\]
After rearranging, we have
\[
    N \le \frac{Cp + Dp}{m_0C(p - 1)} = \frac{C + D}{C} \frac{p}{m_0(p - 1)} \le \frac D C + 1,
\]
contradicting our choice of $N$. Hence $m_0 = n_0$.
\end{proof}

In Example \ref{expl:counter}, it would suffice to take $N = 2$.

By combining Lemma \ref{lem:p_div_coding} and Lemma \ref{lem:p_div_coding2}, we can complete our coding of $\mid_p$ inside $\IF_q\laurent{t}$.

\begin{prop} \label{cor:p_div_coding_final}
Fix an element $\alpha \in \IF_q\laurent{t}$ with valuation $v_t(\alpha) > 0$. Then there exists a parameter $N > 0$, depending on $\alpha$, such that
\[
    n \mid_p m \quad \text{if and only if} \quad m \ge n \wedge \exists a \in \IF_q\laurent{t}\ \alpha^{-mN} - \alpha^{-nN} = a^p - a
\]
holds for all $m,n > 0$.
\end{prop}
\begin{proof}
Write $\alpha = \beta^{p^k}$, $k \in \IN$, such that $\beta$ is not a $p^{\text{th}}$ power in $\IF_q\laurent{t}$. We consider two cases:

\vspace{0.2em}
\textit{Case 1. $p$ does not divide $v_t(\beta)$.} By Lemma \ref{lem:p_div_coding} and Remark \ref{rem:power_of_alpha}, we can choose $N = 1$.

\vspace{0.2em}
\textit{Case 2. $p$ divides $v_t(\beta)$.} By Lemma \ref{lem:p_div_coding2} and Remark \ref{rem:power_of_alpha}, we can choose $N$ to be the smallest natural number not divisible by $p$ bigger than $\hat v_t(\beta^{-1})/v_t(\beta) + 1$.
\end{proof}

From this, we conclude our main theorem.

\begin{thm}
Let $\alpha \in \IF_q\laurent{t}$ be an element with $v_t(\alpha) > 0$. Then the existential theory of the structure $(\IF_q\laurent{t},\alpha,\alpha^{\IZ})$ is undecidable.
\end{thm}
\begin{proof}
First, we identify $\{\alpha^n \mid n > 0\}$ in this structure. This set is given by $\alpha^{\IZ} \cap \IF_q\llbracket t \rrbracket \setminus \{1\}$. In \cite{Anscombe-Koenigsmann}, Anscombe and Koenigsmann show that $\IF_q\llbracket t \rrbracket$ is existentially $\Ll_{\ring}$-definable in $\IF_q\laurent{t}$ without parameters, so the same is true of $\{\alpha^n \mid n > 0\}$ inside $(\IF_q\laurent{t},\alpha,\alpha^{\IZ})$.
By Proposition \ref{cor:p_div_coding_final}, we can interpret $(\IN, 0, 1, +, \mid_p)$ in $(\IF_q\laurent{t},0,1,+,\cdot,\alpha,\alpha^{\IZ})$ using existential formulas. By \ref{idea:1}, $\Th_{\exists}(\IN,0,1,+,\mid_p)$ is undecidable, so $\Th_{\exists}(\IF_q\laurent{t},\alpha,\alpha^{\IZ})$ must also be undecidable.
\end{proof}

\begin{rem} \label{rem:generalise}
Pheidas formulates his theorem in slightly more general terms: for any integral domain $F$ of characteristic $p$, quotient field $K$ of $F$, and intermediate ring $F[t] \subseteq R \subseteq K\laurent{t}$, the existential theory $\Th_{\exists}(R,t,P)$ is undecidable. The same is true of our result: as long as $\alpha \in R$, we have that $\Th_{\exists}(R,\alpha,\{\alpha^n \mid n > 0\})$ is undecidable (essentially by the same proof).
\end{rem}

\begin{rem}
More recently, an adaption of Pheidas' theorem via the so-called Krasner-Kazhdan-Deligne philosophy was obtained by Kartas \cite{Kartas}, who shows that the asymptotic theory of all $p$-adic fields is undecidable in the language of rings with a cross-section. We hope to further adapt these types of results to infinitely ramified valued fields.    
\end{rem}

\paragraph{Acknowledgements.}

First and foremost, I would like to express my gratitude to Philipp Hieronymi who advised this work in 2022.
Moreover, I would like to thank Philip Dittmann, Konstantinos Kartas, and Jochen Koenigsmann for valuable comments that helped improve the exposition of this note.

%%% References %%%
\label{sect:bib}
\bibliographystyle{plain}
\bibliography{main}

\end{document}